\documentclass{amsart}   

\usepackage{tikz,pdfsync,enumerate}

\newcommand{\mC}{{\mathbb C}}
\newcommand{\mR}{{\mathbb R}}
\newcommand{\mN}{{\mathbb N}}

\newcommand{\Om}{\Omega}
\newcommand{\bOm}{b\Omega}
\newcommand{\bD}{b\Delta}

\newcommand{\opA}{\mathcal{A}}

\newcommand{\opC}{\mathcal{C}}
\newcommand{\opE}{\mathcal{E}}
\newcommand{\opH}{\mathcal{H}}
\newcommand{\opI}{\mathcal{I}}
\newcommand{\opK}{\mathcal{K}}
\newcommand{\opL}{\mathcal{L}}
\newcommand{\opM}{\mathcal{M}}
\newcommand{\opS}{\mathcal{S}}

\newcommand{\opT}{\mathcal{T}}

\newcommand{\rar}{\rightarrow}
\newcommand{\ov}{\overline}

\newcommand{\eps}{\varepsilon}
\newcommand{\ds}{\displaystyle}
\newcommand{\Real}{\mathrm{Re\,}}  
\newcommand{\Imag}{\mbox{Im\,}}  
\newcommand{\sze}{Szeg\H{o}}

\newcommand{\perparrow}{\ensuremath{\stackrel{\tiny \perp}{\rightarrow}}}
\newcommand{\eqdef}{\overset{ \text{def} }{=}} 

\newcommand{\vp}{\varphi}

\DeclareMathOperator{\A}{\mathcal A}

\DeclareMathOperator{\F}{\mathcal F}

\newcommand{\R}{\mathbb{R}}

\DeclareMathOperator{\Ran}{Ran}

\DeclareMathOperator{\ess}{ess}
\DeclareMathOperator{\essinf}{\ess\inf}
\DeclareMathOperator{\esssup}{\ess\sup}

\DeclareMathOperator{\Ker}{Ker}

\DeclareMathOperator{\supp}{supp}

\newtheorem{thrm}{Theorem}
\newtheorem{lemma}{Lemma}
\newtheorem{prop}{Proposition}

\begin{document}

\title[Kerzman-Stein operator for piecewise differentiable regions]{The Kerzman-Stein operator for piecewise continuously differentiable regions}
\date{\today}

\author{Michael Bolt}
\thanks{The first author is partially supported by the National Science Foundation under Grant No.~DMS-1002453 and by Calvin College through a Calvin Research Fellowship.}
\address{Department of Mathematics and Statistics\\1740 Knollcrest Circle SE\\Calvin College\\Grand Rapids, Michigan 49546-4403 USA}
\email{mbolt@calvin.edu}

\author{Andrew Raich}
\thanks{The second author is partially supported by  the National Science Foundation under Grant No.~DMS-0855822.}
\address{Department of Mathematical Sciences\\SCEN 327\\1 University of Arkansas\\Fayetteville, Arkansas 72701 USA}
\email{araich@uark.edu}

\begin{abstract}
The Kerzman-Stein operator is the skew-hermitian part of the Cauchy operator defined with respect to an unweighted hermitian inner product on a rectifiable curve.
If the curve is continuously differentiable, the Kerzman-Stein operator is compact on the Hilbert space of square integrable functions; when there is a corner, the operator is noncompact.
Here we give a complete description of the spectrum for a finite symmetric wedge and we show how this reveals the essential spectrum for curves that are piecewise continuously differentiable.
We also give an explicit construction for a smooth curve whose Kerzman-Stein operator has large norm, and we demonstrate the variation in norm with respect to a continuously differentiable perturbation. 
\end{abstract}
\maketitle

%
%

\section{Introduction}
For a smooth region $\Om\subset\subset\mC$, Kerzman and Stein studied in \cite{kerzmanstein} a certain compact operator $\opA$ in relation to the Cauchy projection and \sze\ projection. Let $\opC$ be the Cauchy transform for $\Om$, defined for an integrable function $f$ on $\bOm$ according to
\[
\opC f(z) = \frac{1}{2\pi i} \int_{\bOm} \frac{f(w) \, dw}{w-z} \,\mbox{ for }\, z\in\Om,
\]
and using a nontangential limit of the integral when $z\in\bOm$.  In case the boundary is twice differentiable, it is classical that $\opC$ is a bounded projection from $L^2(\bOm)$ onto the Hardy space $H^2(\bOm)$.  Moreover, the operator $\opA = \opC - \opC^*$ is compact and the \sze\ projection can be written explicitly via $\opS = \opC(\opI+\opA)^{-1}$. The \sze\ projection is the orthogonal projection of $L^2(\bOm)$ to $H^2(\bOm)$. See also Bell~\cite{bell}.

In \cite{lanzani}, Lanzani took important steps to extend the theory to regions with irregular boundary.  In particular, she showed that for regions with continuously differentiable boundary the Kerzman-Stein operator remains compact, and for regions with Lipschitz boundary the Kerzman-Stein equation still holds.  At issue is the extent to which regularity of the boundary leads to a cancellation of singularities in the kernel when $\opA$ is expressed as an integral operator,
\begin{equation}
\opA f(z) =
 \frac{1}{2\pi i} \, P.V. \int_{\bOm}
  \left( \frac{T(w)}{w-z} -\frac{\ov{T(z)}}{\ov w-\ov z} \right)\!
f(w) \, ds_w.
\label{eq-intro1}
\end{equation}
Here $T(w)\in \mC$ is the unit tangent vector at $w\in\bOm$ and $ds = ds_w$ is arc length measure for $\bOm$.   Using \eqref{eq-intro1}, we 
can define the Kerzman-Stein operator even for curves that are not closed. 

The goal of the present work is to give detailed spectral information regarding $\opA$ in order to quantify the extent to which the Cauchy projection fails to be orthogonal.  
Since $\opA$ is compact (and skew-hermitian) for regions with continuously differentiable boundary, its spectrum $\sigma(\opA)$ comprises a countable number of eigenvalues whose only possible accumulation point is $0$. 
If the boundary has a corner, however, then $\opA$ is noncompact as is illustrated by an example of Lanzani~\cite[p.547]{lanzani}.

Our first result describes completely the spectrum for a bounded symmetric wedge.  
\begin{thrm}
\label{thm:R1}   
Let $0<\theta<\pi$ be fixed and $W_\theta = \{ r e^{\pm i\theta}: 0\leq r <1 \}$. If 
\[
\vp(\xi) = \frac{\sinh[\xi(\pi-2\theta)]}{\cosh (\xi\pi)},
\]
then the spectrum of the Kerzman-Stein operator, denoted $\sigma(\A)$, is purely continuous and
\[
\sigma(\A) = i [\inf\vp,\sup\vp] = i [-\sup\vp,\sup\vp] .
\]
\end{thrm}
Interestingly, the spectrum is purely continuous as it was for the example of the unbounded wedge as shown in \cite{bolt00}.  Other examples showing this behavior include an infinite strip and logarithmic sector~\cite{bolt00}. The point spectrum, however, seems to be affected more by the global geometry of the boundary.  In particular, our second result shows that for a piecewise continuously differentiable region, noncompactness occurs precisely due to the individual corners.
\begin{thrm}
Let $\Om\subset\subset\mC$ be a region with (Lipschitz) piecewise continuously differentiable boundary. 
Then the essential spectrum of $\opA$ is the closed interval
\[
\sigma_e(\A) = i [\inf\vp,\sup\vp] = i [-\sup\vp,\sup\vp] .
\]
where $\vp$ in the function in Theorem 1 with $M$ the Lipschitz constant for $\bOm$.
\label{thrmB1} 
\end{thrm}
In particular, we will prove that modulo a compact operator the Kerzman-Stein operator $\opA : L^2(\bOm) \rar L^2(\bOm)$ is equivalent to a diagonal action of Kerzman-Stein operators for symmetric wedges (one wedge for each corner).  The essential spectrum of an operator is unchanged by compact perturbation and so the essential spectrum is the same as that of the diagonal action of Kerzman-Stein operators for symmetric wedges.  Evidently, this is a union of intervals symmetric with respect to zero and is therefore the largest of them.

To better illustrate how the (smooth) global geometry of a region can affect the point spectrum, we include new examples of continuously differentiable regions whose Kerzman-Stein operators have large norm; necessarily their Cauchy projections also have large norm.  For general regions, we also demonstrate how the spectrum changes under continuously differentiable perturbation.  

We mention that Kerzman first suggested the problem of computing the spectrum for $\opA$ in \cite{kerzman}.  Subsequent work on the problem was concerned with giving a complete description of the spectrum for model domains~\cite{bolt00}, asymptotics of eigenvalues for ellipses with small eccentricity~\cite{bolt05}, and norm estimates that are invariant with respect to M\"obius transformation~\cite{barrett-bolt,bolt07}.
For a disc or halfplane, there is complete cancellation of singularities and the Kerzman-Stein operator is trivial~\cite{kerzmanstein}. 

The manuscript is organized as follows.   The spectrum of the Kerzman-Stein 
operator on the symmetric wedge takes the approach of \cite[\S3]{bolt00} and uses a change of variables to convert $\opA$ into a convolution operator. 
The difference in this work is that via some additional Paley-Wiener theory
the Fourier transform produces a rotated Toeplitz operator as opposed to a rotated multiplication operator.
We then construct examples of regions with large Kerzman-Stein eigenvalue.  
Subsequently, we prove that the essential spectrum for a piecewise smooth region is determined only by the angles at the corners, and we modify a method of Lanzani to characterize the change in spectrum due to a continuously differentiable perturbation. 
We conclude with some questions and observations to guide future research.

%
%
\section{Spectrum for a symmetric wedge}

The proof of Theorem \ref{thm:R1} is organized as follows.  Functions in $L^2(W_\theta)$ are represented using pairs of square integrable functions defined on a unit interval that represent the function on the upper and lower segments.  After a change of variable, these pairs correspond with pairs of functions defined on the negative real line. 
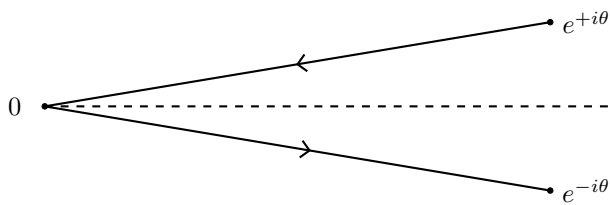
\begin{figure}[ht]
\quad
\begin{tikzpicture}[scale=0.8]

\draw [thick,dashed] (0,0) -- (9.4,0);
\draw [thick] (0,0) -- (8.4,1.4);
\draw [thick] (0,0) -- (8.4,-1.4);

\node at (9,1.4) {$e^{+i\theta}$};
\node at (9,-1.4) {$e^{-i\theta}$};

\node at (-.5,0) {$0$};

\draw [thick] (4.2+ .2*.581,.7+.2*.814) -- (4.2,.7) -- (4.2+ .2*.814,.7+.2*-.581);
\draw [thick] (4.4- .2*.814,-.7333-.2*.581) -- (4.4,-.7333) -- (4.4- .2*.581 , -.7333-.2*-.814);

\draw[fill] (0,0) circle (.045);
\draw[fill] (8.4,1.4) circle (.045);
\draw[fill] (8.4,-1.4) circle (.045);

\end{tikzpicture}
\caption{Symmetric wedge, $W_\theta = \{ r e^{\pm i\theta}: 0\leq r \leq 1 \}$}
\label{fig-wedge}
\end{figure}
 In the new variable, the Kerzman-Stein operator acts as a 
convolution operator on $\mR$ that is conjugated with multiplication by the characteristic function for the negative real numbers.  The spectrum is analyzed using Paley-Wiener theory and standard results about Toeplitz operators. 
The symmetric wedge is illustrated in Figure~\ref{fig-wedge}. 

\smallskip
So let $g = \A f$ where $f,g\in L^2(W_\theta)$ are expressed using 
\[
f = \begin{pmatrix} f_+ \\ f_- \end{pmatrix}
\quad\text{and}\quad
g = \begin{pmatrix} g_+ \\ g_- \end{pmatrix}
\]
where $f_\pm, g_\pm \in L^2([0,1])$ and 
the identification $L^2(W_\theta) \cong L^2([0,1]) \times L^2([0,1])$ is made
via   $f_\pm(s) = f(s e^{\pm is})$, $g_\pm(s) = g(s e^{\pm is})$.
Then using \eqref{eq-intro1}, 
\begin{align*}
g_{\pm}(s) &= \frac{\chi_{[0,1]}(s)}{2\pi i}
  \int_0^1 \Big( \frac{\pm e^{\mp i\theta}}{t e^{\mp i\theta}-s e^{\pm i\theta}} 
   - \frac{\mp e^{\mp i\theta}}{t e^{\pm i\theta}-s e^{\mp i\theta}}\Big) f_{\mp}(t)
   \, dt\\
      &= \frac{\pm e^{\mp i\theta}\cos\theta \, \chi_{[0,1]}(s)}{\pi i}
      \int_0^1 \frac{t-s}{t^2+s^2-2st\cot(2\theta)} f_{\mp}(t)\, dt
\end{align*}
for $s\in (0,1)$.  (Since the functions are defined in $L^2([0,1])$, the values at the endpoints may be ignored.)

The isometry $\Lambda: L^2([0,1]) \to L^2((-\infty,0])$ defined by $\Lambda h(u) = h(e^u)e^{u/2}$ for $u<0$ enables us to rewrite this as
\begin{align}
(\Lambda g_\pm)(u) 
&=  \chi_{(-\infty,0)}(u)\int_{-\infty}^0 k_\mp (u-v) (\Lambda f_{\mp})(v)\, dv  \notag \\[.03in]
&=  \chi_{(-\infty,0)}(u) \big(k_\mp*(\chi_{(-\infty,0)}  \Lambda f_{\mp})\big)(u) \label{eq:B0}
\end{align}
where
\[
k_\mp(u) = \frac{\mp e^{\mp i\theta}\cos(\theta)}{\pi i} \frac{\sinh(u/2)}{\cosh(u)-\cos 2\theta}.
\]

We next recall one of the classical Paley-Wiener theorems as can be found in Rudin~\cite{Rud87}.  
Let $\mathcal F f = \widehat f$ denote the Fourier transform as given by 
\[
\widehat f(\xi) = \frac{1}{\sqrt{2\pi}} \int_\R e^{-ix\xi} f(x)\, dx
\]
valid for $\Imag \xi > 0$ and 
a.e.\  $\xi \in \mR$.  
Also let $H^2(\mR)\subset L^2(\mR)$ denote the (closed) Hardy subspace consisting of functions that are holomorphic in the upper half plane and whose boundary values are square integrable. 
Then a Paley-Wiener Theorem \cite[Theorem 19.2]{Rud87} says that the Fourier transform is an isometry $L^2((-\infty,0))\to H^2(\mR)$.   
Of course, the inner product on $H^2(\mR)$ is the one inherited from $L^2(\mR)$.
Stated differently, Theorem 19.2 says that via the Fourier transform, multiplication by $\chi_{(-\infty,0)}$ in $L^2(\mR)$ corresponds with taking the \sze\ projection $L^2(\mR) \perparrow H^2(\mR)$.  That is, 
\begin{equation}
\mathcal F\big(\chi_{(-\infty,0)}f\big)(\xi) 
= \opS \widehat f(\xi) 
\label{eqn:Ft of (-infty,0]}.
\end{equation}
for $f\in L^2(\mR)$.

We now define
\[
\vp(\xi) =  \frac{\sinh[\xi(\pi-2\theta)]}{\cosh (\xi\pi)}  
\]
for $\xi\in\mR$.
Then by formula \cite[3.984-3]{GrRy07}, it follows that $\widehat k_- = +e^{-i \theta} \vp / \sqrt{2\pi}
 $ and $ \widehat k_+= -e^{+i \theta} \vp / \sqrt{2\pi}$. 
By also incorporating \eqref{eq:B0} and \eqref{eqn:Ft of (-infty,0]}, it follows that the Kerzman-Stein operator for the symmetric wedge can be expressed succinctly using
\[
\begin{pmatrix}
\widehat{\Lambda g_+} \\ \widehat{\Lambda g_-}
\end{pmatrix}
= \begin{pmatrix}
+e^{-i\theta}  \opS \big[ \vp \widehat{\Lambda f_-} \big] \\
-e^{+i\theta}  \opS \big[ \vp \widehat{\Lambda f_+} \big]
\end{pmatrix}.
\]
By defining a Toeplitz operator $\opT = \opS \vp: H^2(\mR) \rar H^2(\mR)$ and
by utilizing the isometry $\mathcal F \Lambda: L^2([0,1]) \rar H^2(\mR)$, we can further simplify the problem by noting that 
$\sigma(\opA) = \sigma(\widetilde \opA)$ where
$
\widetilde\opA =  ( \F\Lambda)\circ\opA\circ(\F\Lambda)^{-1}
$ 
acts on $H^2(\mR) \times H^2(\mR)$ via 
\[
\widetilde \opA \begin{pmatrix} h_1 \\ h_2 \end{pmatrix} = \begin{pmatrix} 0\! & \!+e^{-i\theta}\opT \\ -e^{+i\theta}\opT\! &\! 0\end{pmatrix}
\begin{pmatrix} h_1 \\ h_2 \end{pmatrix}. 
\]
Since $\opA$ is skew-hermitian, it is immediate that its spectrum is imaginary.  In fact, the spectrum is symmetric with respect to zero since it commutes with the anti-linear involution, $f\in L^2 \rar \ov {fT}$.  (Recall that $T=T(w)$ is the unit tangent vector at $w\in W_\theta$ represented as a complex number.)  This fact is established in \cite{bolt05}.

Since $\vp$ is smooth, bounded, and real-valued, $\opT$ is self-adjoint.
Its spectrum is given according to the theorem of Hartman and Wintner by 
\[
\sigma(\opT) = [\essinf \vp, \esssup \vp]  = [\inf \vp, \sup \vp] = [-\sup\vp, +\sup\vp]
\]
where the last equality follows since $\vp$ is odd.  For a proof of the Hartman and Winter result, see Douglas~\cite[Theorem~7.20]{douglas}.  We mention that a slight modification is needed to accommodate the application needed here.  As presented in \cite{douglas}, the result applies to Toeplitz operators for the unit disc, $\opS \vp: H^2(\bD) \rar H^2(\bD)$.
That the result holds also for the upper half-plane can be seen via the isometry $\Lambda_h : L^2(\bD) \rar L^2(\mR)$ given by $\Lambda_h f = (f\circ h) \sqrt{h'}$ for $h(z) = (z-i)/(z+i)$.  The isometry commutes with the \sze\ projection and is induced by the biholomorphism $h : \{z\in \mC: \Imag z > 0\} \rar  \{z\in \mC: |z| < 1\}$.

We show next that $\sigma(\opT)$ is purely continuous.  As for any bounded operator on a Hilbert space, the spectrum of $\opT : H^2(\mR)\rar H^2(\mR)$ 
can be decomposed as 
\[
\sigma(\opT) = \sigma_p(\opT) \cup \sigma_c(\opT) \cup \sigma_r(\opT) 
\]
where the point spectrum $\sigma_p(\opT)$ consists of $\lambda \in \mC$ for which $\opT - \lambda \opI$ is not one-to-one, the continuous spectrum $\sigma_c(\opT)$ consists of $\lambda\in\mC$ for which  $\opT - \lambda \opI$ is a one-to-one mapping onto a dense proper subspace of $H^2(\mR)$, and the residual spectrum $\sigma_r(\opT)$ consists of all other $\lambda\in\sigma(\opT)$.  (See Rudin~\cite[p.343]{Rud91}.)
The fact that $\opT$ has no point spectrum or residual spectrum follows from the claims:
\vspace{.04in}
\begin{enumerate}
\item[(i)] $\lambda\in\sigma(\opT)$ implies that $\opT -\lambda \opI$ is one-to-one, and
\vspace{.05in}
\item[(ii)] the range of $\opT - \lambda \opI$ is dense in $H^2(\mR)$.
\end{enumerate}
\vspace{.04in}
The proof of (i) is a direct consequence of a proposition by Coburn.  In particular, since $\vp-\lambda \in L^\infty(\mR)$ for any fixed $\lambda\in\mR$, then $\Ker( \opT - \lambda \opI) = \{0\}$.  For Coburn's result, see Douglas~\cite[Proposition~7.24]{douglas}.  As in the preceding paragraph, the application needs the equivalent interpretation for the upper half-plane.  
It also uses the fact that  $\opT-\lambda \opI$ is a self-adjoint Toeplitz operator with symbol $\vp-\lambda$.
To prove (ii), suppose $\lambda\in\sigma_r(\opT)$.  Then $\opT - \lambda \opI$ is one-to-one and does not have dense range; in particular, $\dim\left[\Ran (\opT-\lambda \opI)\right]^\perp \geq 1$.  However, 
\[
\large[\Ran (\opT-\lambda \opI)\large]^\perp =  \Ker(\opT-\lambda \opI)^*
 = \Ker(\opT^* - \lambda \opI) = \Ker(\opT -\lambda \opI),
\] 
implying that $\lambda\in\sigma_p(\opT)$.  So (ii) follows from (i).

At last we come to the proof of the claims in Theorem~\ref{thm:R1}.  
We first establish that $\sigma(\widetilde\opA) = i \, \sigma(\opT)$.  To determine when $\pm i\lambda \in \sigma(\widetilde \opA)$, we consider the invertibility of 
$ \widetilde \opA\pm i\lambda I$.
Using
\begin{equation}
(\widetilde\opA \pm i \lambda \opI)^{-1} = (\opT-\lambda \opI)^{-1} (\opT+\lambda \opI)^{-1}
(\widetilde\opA^* \pm i\lambda \opI),
\label{eq:act-diag}
\end{equation}
we see that $\pm \lambda \not\in \sigma(\opT)$ implies $\pm i\lambda\not\in\sigma(\widetilde \opA)$.  So $\sigma(\widetilde\opA) \subset i \,\sigma(\opT)$.  
It is to be understood in \eqref{eq:act-diag} that $(\opT- \lambda \opI)^{-1}$ and $(\opT + \lambda \opI)^{-1}$ act diagonally on $H^2(\mR) \times H^2(\mR)$.
Similarly, using
\begin{equation}
(\opT^2-\lambda^2 \opI)^{-1}
=  
(\widetilde\opA +i\lambda \opI)^{-1}(\widetilde\opA -i\lambda \opI)^{-1},
\label{eq:act-diag2}
\end{equation}
we observe that $\pm i \lambda \not\in \sigma(\widetilde\opA)$ implies $\lambda^2 \not\in \sigma(\opT^2)$, and therefore $\pm \lambda \not\in \sigma(\opT)$.   
Consequently, $i \,\sigma(\opT) \subset \sigma(\widetilde\opA)$.
Again it is to be understood in \eqref{eq:act-diag2} that $(\opT^2 - \lambda \opI)^{-1}$ acts diagonally.
It follows that $\sigma(\widetilde\opA) = i \, \sigma(\opT)$ as claimed.

It remains to be seen that $\sigma (\widetilde \opA)$ is purely continuous.  As done in the earlier paragraph, it will suffice to establish the following claims:
\vspace{.04in}
\begin{enumerate}
\item[(i)] $\lambda\in\sigma(\widetilde\opA)$ implies that $\widetilde\opA -\lambda \opI$ is one-to-one, and
\vspace{.05in}
\item[(ii)] the range of $\widetilde\opA - \lambda \opI$ is dense in $H^2(\mR) \times H^2(\mR)$.
\end{enumerate}
\vspace{.04in}
To prove (i), suppose $(h_1,h_2) \in H^2(\mR)\times H^2(\mR)$ and $(\widetilde\opA - \lambda \opI)(h_1, h_2)=0$.  Then
\begin{align*}
-\lambda h_1 + e^{-i\theta} \opT h_2 &= 0 \\
-e^{+i\theta}\opT h_1 - \lambda h_2 &=0.
\end{align*}
From the first equation it follows that $e^{+i\theta} \lambda h_1 = \opT h_2$, so  the second equation gives $-(\opT^2 + \lambda^2)h_2 =0$ after multiplying by $\lambda$.
It follows that $h_2 \equiv 0$ or else $\pm i \lambda \in \sigma_p(\opT)$.  The latter case is excluded by (i) from the earlier paragraph.  Hence, $h_2 \equiv 0$ which implies $h_1 \equiv 0$ by the equations above, and (i) is proved.
To prove (ii), we again notice that 
\[
\large[\Ran (\widetilde \opA-\lambda \opI)\large]^\perp =  \Ker(\widetilde\opA-\lambda \opI)^*
 = \Ker(\widetilde \opA^* - \lambda \opI) = -\Ker(\widetilde\opA +\lambda \opI).
\] 
Therefore, if $\Ran(\widetilde\opA-\lambda \opI)$ is not dense, then $\widetilde\opA +\lambda \opI$ has a nontrivial kernel and (ii) again follows from (i).


%
%

\section{Smooth regions with large Kerzman-Stein eigenvalue}
In this section we construct continuously differentiable regions whose Kerzman-Stein operators have a large eigenvalue.  
The regions also provide examples where the Cauchy operator has large norm.  It will be apparent in the argument that there exist infinitely smooth regions that satisfy the same estimates. 

We begin by considering a configuration of $n+1$ unit intervals as shown in Figure~\ref{figBex1}.  The intervals are equally spaced and have a total width separation of $\eps$. 
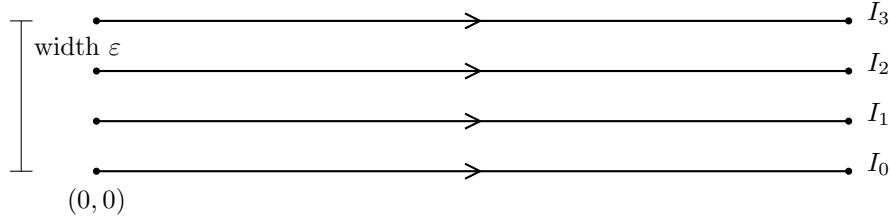
\begin{figure}[ht]
\quad
\begin{tikzpicture}

\draw [thick] (0,0) -- (10,0);
\draw [thick] (0,.666) -- (10,.666);
\draw [thick] (0,1.333) -- (10,1.333);
\draw [thick] (0,2) -- (10,2);

\node at (0,-.4) {$(0,0)$};

\node at (10.4,.1) {$I_0$};
\node at (10.4,.766) {$I_1$};
\node at (10.4,1.433) {$I_2$};
\node at (10.4,2.1) {$I_3$};

\draw (-1,0) -- (-1,2);
\draw (-1.15,0) -- (-.85,0);
\draw (-1.15,2) -- (-.85,2);

\node at (-.25,1.667) {width $\eps$};

\draw [thick] (4.9,-.1) -- (5.1,0) -- (4.9,.1);
\draw [thick] (4.9,.566) -- (5.1,.666) -- (4.9,.766);
\draw [thick] (4.9,1.233) -- (5.1,1.333) -- (4.9,1.433);
\draw [thick] (4.9,1.9) -- (5.1,2) -- (4.9,2.1);

\draw[fill] (0,0) circle (.04);
\draw[fill] (10,0) circle (.04);
\draw[fill] (0,.666) circle (.04);
\draw[fill] (10,.666) circle (.04);
\draw[fill] (0,1.333) circle (.04);
\draw[fill] (10,1.333) circle (.04);
\draw[fill] (0,2) circle (.04);
\draw[fill] (10,2) circle (.04);

\end{tikzpicture}
\caption{Configuration of unit intervals when $n=3$}
\label{figBex1}
\end{figure}
The intervals are denoted $I_j$, $0\leq j \leq n$, and are oriented as shown.
The union of the intervals forms part of the boundary a region $\Om_{n,\eps}$ as illustrated in Figure~\ref{figBex2}.  The remaining boundary pieces are unit intervals and semicircles, and they are assembled
\begin{figure}[ht]

\begin{tikzpicture}

\draw [thick,dashed,fill=gray] (0,2.333) arc (90:270:1.167);
\draw [thick,dashed,fill=gray] (10,0) arc (-90:90:.167);
\draw [thick,dashed,fill=gray] (10,.666) arc (-90:90:.167);
\draw [thick,dashed,fill=gray] (10,1.333) arc (-90:90:.167);
\draw [thick,dashed,fill=gray] (10,2) arc (-90:90:.167);

\draw [thick,dashed,fill=white] (0,.666) arc (90:270:.167);
\draw [thick,dashed,fill=white] (0,1.333) arc (90:270:.167);
\draw [thick,dashed,fill=white] (0,2) arc (90:270:.167);

\path [fill=white] (-.02,.353) -- (.02,.353) -- (.02,.646) -- (-.02,.646);
\path [fill=white] (-.02,1.02) -- (.02,1.02) -- (.02,1.353) -- (-.02,1.353);
\path [fill=white] (-.02,1.686) -- (.02,1.686) -- (.02,1.98) -- (-.02,1.98);

\path [fill=gray] (-.02,0) -- (10.02,0) -- (10.02,.333) -- (-.02,.333);
\path [fill=gray] (-.02,.666) -- (10.02,.666) -- (10.02,1) -- (-.02,1);
\path [fill=gray] (-.02,1.333) -- (10.02,1.333) -- (10.02,1.666) -- (-.02,1.666);
\path [fill=gray] (-.02,2) -- (10.02,2) -- (10.02,2.333) -- (-.02,2.333);

\draw [thick] (0,0) -- (10,0);
\draw [thick,dashed]  (0,.333) -- (10,.333);
\draw [thick] (0,.666) -- (10,.666);
\draw [thick,dashed] (0,1) -- (10,1);
\draw [thick] (0,1.333) -- (10,1.333);
\draw [thick,dashed] (0,1.666) -- (10,1.666);
\draw [thick] (0,2) -- (10,2);
\draw [thick,dashed] (0,2.333) -- (10,2.333);

\draw [thick] (4.9,-.1) -- (5.1,0) -- (4.9,.1);
\draw [thick] (4.9,.566) -- (5.1,.666) -- (4.9,.766);
\draw [thick] (4.9,1.233) -- (5.1,1.333) -- (4.9,1.433);
\draw [thick] (4.9,1.9) -- (5.1,2) -- (4.9,2.1);

\draw[fill] (0,0) circle (.04);
\draw[fill] (10,0) circle (.04);
\draw[fill] (0,.666) circle (.04);
\draw[fill] (10,.666) circle (.04);
\draw[fill] (0,1.333) circle (.04);
\draw[fill] (10,1.333) circle (.04);
\draw[fill] (0,2) circle (.04);
\draw[fill] (10,2) circle (.04);

\end{tikzpicture}
\caption{Continuously differentiable region $\Om_{n,\eps}$ when $n=3$}

\label{figBex2}

\end{figure}
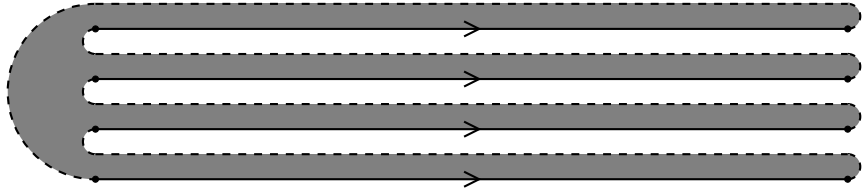
 in such a way that the region has continuously differentiable boundary. 
 (In fact, this construction gives a region of class $C^{1,1}$.  It would be an easy matter to construct a 
 region of class $C^\infty$ that also has the intervals as part of the boundary.)
Since the (skew-hermitian) Kerzman-Stein operator $\opA$ for $\Om_{n,\eps}$ is then compact, in order to show it has a large eigenvalue, it will be enough to show that $\|\opA \|$ is large.
For this we will show that $\| \opA f_0\|$ is large for a particular function $f_0$ with $\| f_0 \| =1$.  
The function we have in mind takes value $1$ on $I_0$ and value $0$ elsewhere.  To show that $\|\opA f_0\|$ is large, it will be enough to show that the restriction $\opA f_0$ to the union $\cup_{j=1}^n I_j $ has large norm.  

In this way, the problem is reduced to considering the initial configuration of unit intervals.
For $z\in I_j$, $w\in I_k$ we have
\begin{align*}
z = x+ i\,\eps j/n\\
w = y+ i\,\eps k/n
\end{align*}
where $x,y\in[0,1]$. Moreover, regardless of $j,k$ we have $T(z) = 1 = T(w)$.  A simple computation reveals for the Kerzman-Stein kernel,
\begin{align*}
A(z,w) &= \frac{1}{2\pi i} 
\left(
\frac{1}{(y-x) + i\,\eps(k-j)/n} 
-
\frac{1}{(y-x) -i\,\eps(k-j)/n} 
\right) \\
&= \frac{1}{\pi}\,
\frac{\eps(j-k)/n}{(y-x)^2 +\eps(j-k)^2/n^2}.
\end{align*}
As expected, the kernel vanishes when $z,w$ belong to the same interval; that is, when $j=k$.

Introducing notation $\alpha_j = \eps j/n$, and then writing $z= x+i\,\alpha_j$ for a point of $I_j$, $j\neq 0$, we find for the function $f_0$ described above,
\begin{align}
\opA f_0(x+ i\,\alpha_j) &= \int_{I_0} A(z,w) f_0(w) \, ds_w \notag \\
& = \frac{1}{\pi} \int_0^1 \frac{\alpha_j}{(y-x)^2 + \alpha_j^2}  \, dy \notag \\
& = \frac{1}{\pi} 
\left(
\arctan \frac{1-x}{\alpha_j}
+ \arctan \frac{x}{\alpha_j}
\right)\!. \label{eqBex1}
\end{align}
We define $g_j(x) = \opA f_0(x+ i\, \alpha_j)$ for $x\in[0,1]$ using the last expression \eqref{eqBex1}.  Basic calculus can be used to show that $g_j(x) = g_j(1-x)$ and that $g_j$ increases on the interval $[0,1/2]$. It follows that for $x\in[\alpha^{1/2}_j,1-\alpha^{1/2}_j]$,
\[
g_j(x) 
\geq \frac{1}{\pi} 
\left(
\arctan \frac{1- \alpha^{1/2}_j}{\alpha_j}
+ \arctan \frac{1}{\alpha^{1/2}_j}
\right)
\geq 
 \frac{2}{\pi} 
\,\arctan \frac{1}{\alpha^{1/2}_j}.
\]
The latter estimate uses $(1-\alpha_j^{1/2})/\alpha_j \geq 1/\alpha_j^{1/2}$ which is valid when $\eps <1/4$ since then $\alpha_j <1/4$.  Using basic calculus 
we observe $\arctan (1/x) > \pi/2 - x$ for $x>0$, so 
\begin{align*}
\| \opA f_0 \|^2 &\geq \sum_{j=1}^n \int_0^1 (g_j(x))^2 \, dx\\
& \geq \sum_{j=1}^n (1-2\alpha^{1/2}_j) \, \frac{4}{\pi^2} \left(\frac{\pi}{2}-\alpha^{1/2}_j\right)^2 \\
& \geq \sum_{j=1}^n \left[1 -  \left(2+\frac{4}{\pi} \right) \alpha^{1/2}_j\right] \\
& = n - \left(2+\frac{4}{\pi}\right)  \sum_{j =1}^n \sqrt\frac{\eps j}{n}
\end{align*}
where the third inequality holds for small positive $\alpha_j$.  (Certainly $\alpha_j < 1/4$ is sufficient.)  Finally, using a basic lower (Riemann) sum we estimate
\[
\sum_{j=1}^n \sqrt\frac{j}{n}
=
1+ \sum_{j=1}^{n-1} \left(\sqrt\frac{j}{n} \, \frac{1}{n}\right) \cdot n
<1 +  \int_0^1 \sqrt x \, dx \cdot n
= 1+ \frac{2n}{3}
\] 
so that 
\begin{align*}
\| \opA \|^2 \geq 
\| \opA f_0 \|^2 &> 
n -  \left(2+\frac{4}{\pi}\right) \sqrt \eps \left(1+ \frac{2n}{3}\right) \\
&= n \left(1 -  \frac{4\pi + 8}{3\pi} \sqrt\eps \right) -  \frac{2\pi + 4}{\pi} \sqrt \eps.
\end{align*}
This can be made arbitrarily large by choosing $n\in \mN$ suitably large
and $\eps>0$ suitably small.

We mention how this also can be used to give an estimate for the norm of the Cauchy projection.  For this, recall that the Cauchy transform $\opC$ defined using a nontangential limit (as in \S1) is related by the Plemelj formulas to the Cauchy singular operator $\opC_0$ defined on $\bOm$ by
 \[
\opC_0 f(z) = \frac{1}{2\pi i} \, P.V. \int_{\bOm} \frac{f(w) \, dw}{w-z} \,\mbox{ for a.e. }\, z\in\bOm.
\]
In particular,  $\opC = \opC_0 + \frac{1}{2}\opI $ where $\opI$ is the identity. (This is for a region with continuously differentiable boundary.)  It is classical that $\opC$ and $\opC_0$ are bounded on $L^2(\bOm)$, so evidently the Kerzman-Stein operator $\opA \eqdef \opC-\opC^* =  \opC_0-\opC_0^*$ is bounded and can be expressed using \eqref{eq-intro1}.  Notice that the holomorphic and arc length differentials are related via $dw = T(w) \, ds_w$.

By the Cauchy integral formula, it is immediate that $\opC$ projects $L^2(\bOm)$ to the closed subspace $H^2(\bOm)$ consisting of limiting values of functions holomorphic on $\Om$.  
The norms of $\opA$ and $\opC$ as operators on $L^2(\bOm)$ are related then via 
\[
\| \opA \| = \left(\| C\|^2-1\right)^{1/2}.
\]
In fact, this identity holds for general projection operators acting on a Hilbert space.  (A proof can be found in Gerisch~\cite{gerisch}. 
Notice that our definitions give $\opA = 2i\,  \Imag\opC$.)

In case $\Om=\Om_{n,\eps}$ it follows that 
\[
\| \opC\|^2   \geq n \left(1 -  \frac{4\pi + 8}{3\pi} \sqrt\eps \right) +1 -  \frac{2\pi + 4}{\pi} \sqrt \eps.
\]
Notice, in particular, the curious limiting behavior as $\eps\rar 0^+$ which says 
\[
\liminf_{\eps\rar 0^+} \|\opC\|^2 \geq n+1
\]
where the right-hand side is exactly the number of intervals from which $\Om_{n,\eps}$ is constructed. 
Related to this, we mention work of Feldman, Krupnik, and Spitovsky that gives the norm of the Cauchy singular operator for a finite family of parallel lines~\cite{feldman}.   In particular, they prove 
$\| \opC_0 \| = \frac{1}{2}\cot \frac{\pi}{4n}$ where $n$ is the number of lines.

\medskip

We mention that throughout this manuscript, $\| \cdot \|$ represents an $L^2$ norm.  Other norms will be indicated using subscripts; for instance, $\| \cdot\|_\infty$ is an $L^\infty$ norm.

%
%

\section{Kerzman-Stein operator for a piecewise continuously differentiable region}

The proof of Theorem~\ref{thrmB1} is organized as follows.  We use characteristic functions to isolate the behavior of the Kerzman-Stein operator to various pieces of the boundary.  Locally near a corner, we use the following lemma to show compactness for the difference between Kerzman-Stein operators for $\bOm$ and an approximating symmetric wedge. (A proof can be found, for instance, in D'Angelo~\cite[p.148]{d'angelo}.)  
\begin{lemma}
Let $L : H_1 \rar H_2$ be a linear operator between Hilbert spaces.  Then $L$ is compact if and only if for every $\eps>0$ there exists $C = C_\eps>0$ and a compact operator $K = K_\eps$ such that 
\[ \| Lz \| \leq \eps \| z \| + C \| Kz \| \]
for all $z\in H_1$.
\label{lemB2}
\end{lemma}
Locally away from a corner, compactness follows from the Lanzani result that says the Kerzman-Stein operator for the boundary of a continuously differentiable region is compact~\cite{lanzani}. 
 For the non-local behavior, we show compactness using the theory of Hilbert-Schmidt integral operators.  

\subsection{Localization on the boundary via characteristic functions.} 
We begin by setting up characteristic functions to make precise the various pieces.
For each corner $w_j\in \bOm$, $1\leq j\leq n$, take an open set $U_j$ containing $w_j$ such that $\bOm \cap \ov U_j$ can be expressed as a connected graph over a symmetric interval.
  In particular, after a translation and a rotation, the set 
  $\bOm\cap \ov U_j$ is a graph $w = x + i\phi_j(x)$ over an interval $x\in[-a_j,+a_j]$ and the corner is located at the origin.
The rotation is such that $\phi_j(x) = M_j |x| + p_j(x)$ where $M_j\neq 0$ and $p_j\in C^1([-a_j,+a_j])$ with $p_j(0) = 0 = p_j'(0)$.   
The $U_j$, $1\leq j\leq n$, are assumed to be small enough that their closures are pairwise disjoint.   
 Let $\chi_j$ be the characteristic function (defined on $\bOm$) for the image set of the graph $w = x  +i\phi_j(x)$ when $x \in [-a_j,+a_j]$.  
 Away from the image set, $\chi_j$ takes value zero.
 Let one more characteristic function $\chi_0$ on $\bOm$ be defined according to $\chi_0 = 1 - \sum_{j=1}^n \chi_j$.

We express the Kerzman-Stein operator as a sum
$\opA  = \sum_{j,k = 0}^n \opA_{j,k}$
where $\opA_{j,k}$ is the integral operator associated to the kernel 
\[
A_{j,k}(z,w) = \frac{1}{2\pi i} 
\left( \frac{T(w)}{w-z} - \frac{\ov{T(z)}}{\ov w-\ov z} \right) \chi_j(z) \chi_k(w).
\]
Each $\opA_{j,k}$ is bounded on $L^2(\bOm)$, for if  
 $\opM_j$ denotes multiplication by $\chi_j$, then 
$\opA_{j,k} = \opM_j \circ \opA \circ \opM_k$.
 So the boundedness of $A_{j,k}$ 
 follows from the boundedness of $\opA$ and the boundedness of each $\opM_j$.

The kernels $A_{j,j}(\cdot,\cdot)$, $1\leq j \leq n$, have compact support in $(\bOm\cap \ov U_j) \times(\bOm\times \ov U_j)$ and so $\opA_{j,j}$ takes functions supported on $\bOm\cap \ov U_j$ to functions supported on $\bOm\cap \ov U_j$.   In what remains we show first that $\opA_{j,j}$ differs from the Kerzman-Stein operator for a symmetric wedge only by a compact operator. 
As a consequence, $\sum_{j=1}^n \opA_{j,j}$ is equivalent (spectrally, modulo a compact operator) to a diagonal action of Kerzman-Stein operators for symmetric wedges.  
We then show that the remaining operators $\opA_{j,k}$ are compact.   Taken together, these observations show that $\opA  = \sum_{j,k = 0}^n \opA_{j,k}$ 
is equivalent (spectrally, modulo a compact operator) to a diagonal action of Kerzman-Stein operators for symmetric wedges.

\subsection{Proof that $\opA_{j,j}$, $1\leq j\leq n$, is equivalent to the Kerzman-Stein operator for a symmetric wedge.} 

The proof of this fact is exactly the following approximation result whose latter claim is based on Lemma~\ref{lemB2}.  For ease of notation, we drop subscripts for the remainder of the subsection.

\begin{prop}
Let $p \in C^1([-a,+a ])$ be such that $p(0) = p'(0) = 0$ and let $M\neq 0$.   Consider two curves expressed as graphs for $x\in[-a , +a]$ via
\begin{align*}
\Gamma_{M,p} & : z_p(x)  =  x+ i( M|x| +p(x))  \\ 
\Gamma_{M,0} & : z_0(x)  =  x+ i  M|x|  
\end{align*} 
Expressed as an integral operator in terms of the graph parameter, the difference in the respective Kerzman-Stein operators is
$ \opA_{M,p} - \opA_{M,0} =
 \opE^{--}_{M,p} + \opE^{-+}_{M,p}+\opE^{+-}_{M,p}+ \opE^{++}_{M,p}$
 where $\opE^{--}_{M,p}$, $\opE^{++}_{M,p}$ are compact, and where
\begin{align}
\| \opE^{-+}_{M,p} \|  &< 
\|p'\|_\infty \, 3 \,(1+(|M|+\| p' \|_\infty)^2)^{1/2} 
 \label{eqA01}  \\
\| \opE^{+-}_{M,p} \|  &< 
\|p'\|_\infty \, 3 \,(1+(|M|+\| p' \|_\infty)^2)^{1/2}.
\label{eqA02}  
 \end{align}
 In particular, $\opA_{M,p} - \opA_{M,0}$ is compact.
 \label{propB3}
\end{prop}
\begin{proof}
Following the notation in Lanzani~\cite{lanzani}, let  $\phi(x) = M|x| + p(x)$ and 
\[ h(x)= (1+\phi'(x)^2)^{1/2}. \] 
The isometry $\Lambda : L^2(\Gamma_{M,p}) \rar L^2( [-a,+a])$ given by
\[
(\Lambda f) (x) = f(x + i\phi(x))\, h(x)^{1/2}
\]
enables us to express the Kerzman-Stein operator for $\Gamma_{M,p}$ as an integral operator
on $L^2([-a,+a])$ with kernel
\begin{align*}
A_{M,p}&(x,y) \\  &=  \frac{1}{2\pi i} \frac{1}{h(x)^{1/2} h(y)^{1/2}} 
\left(
 \frac{h(x)[1+i\phi'(y)] }{(y-x) + i[\phi(y) - \phi(x)]} -
 \frac{h(y)[1-i\phi'(x)] }{(y-x) - i[\phi(y) - \phi(x)]} 
 \right) \! .
\end{align*}
By letting $p=0$ in the previous sentences, we likewise express the Kerzman-Stein operator for $\Gamma_{M,0}$ as an integral operator with kernel $A_{M,0} (x,y)$. 

The operators $\opE_{M,p}^{\pm\pm}$ described in the proposition are then defined in terms of their kernels,
\begin{align*}
E^{--}_{M,p}(x,y) = \left(   A_{M,p} (x,y)-A_{M,0} (x,y)    \right)  \chi_{(-a,0)}(x)  \chi_{(-a,0)}(y) \\
E^{-+}_{M,p}(x,y) = \left(   A_{M,p} (x,y)-A_{M,0} (x,y)    \right)  \chi_{(-a,0)}(x)  \chi_{(0,+a)}(y) \\
E^{+-}_{M,p}(x,y) = \left(   A_{M,p} (x,y)-A_{M,0} (x,y)    \right)  \chi_{(0,+a)}(x)  \chi_{(-a,0)}(y) \\
E^{++}_{M,p}(x,y) = \left(   A_{M,p} (x,y)-A_{M,0} (x,y)    \right)  \chi_{(0,+a)}(x)  \chi_{(0,+a)}(y) . \!
\end{align*}
The notation is meant to reflect, for example, that $\opE^{+-}_{M,p}$ takes a function supported on the half-interval $(-a,0)$ to a function supported on the half-interval $(0,+a)$.

In case that $x,y$ are both negative or both positive, then $z_0(x)$ and $z_0(y)$ trace the same line segment and so $A_{M,0}(x,y) = 0$.   This reflects the observation that the Kerzman-Stein operator vanishes for a disc or halfplane.
It follows that $\opE^{--}_{M,p}$ and $\opE^{++}_{M,p}$ act exactly as the Kerzman-Stein operators restricted to $\Gamma_{M,p}^- = \Gamma_{M,p} \cap \{\Real z <0\}$ and $\Gamma_{M,p}^+= \Gamma_{M,p} \cap \{\Real z >0\}$, respectively.
Since $\Gamma^-_{M,p}$ and $\Gamma^+_{M,p}$ are continuously differentiable, it follows that $\opE^{--}_{M,p}$ and $\opE^{++}_{M,p}$ are compact.  So the first claim in the proposition is established.

For the remaining claims, a simple check of the kernels reveals $(\opE^{+-}_{M,p})^* = - \opE^{-+}_{M,p}$.  So it suffices to verify \eqref{eqA01}.   Then \eqref{eqA02} follows immediately.

We proceed to express the kernel $E_{M,p}^{-+}$ as a sum of manageable parts.   
The kernel vanishes except for $x\in(-a , 0)$, $y \in(0, +a)$, and in this region we have
\begin{align*}
2\pi i \,E^{-+}_{M,p}(x,y) &= 
\left(
 \frac{h(x)^{1/2}}{h(y)^{1/2}} \frac{1+i\phi'(y) }{(y-x) + i(\phi(y) - \phi(x))} -
\frac{1+iM}{(y-x) + iM(y+x)}
 \right) \\
 & -\left(
 \frac{h(y)^{1/2}}{h(x)^{1/2}} \frac{1-i\phi'(x) }{(y-x) - i(\phi(y) - \phi(x))} -
\frac{1+iM}{(y-x) - iM(y+x)}
 \right)
 \!.
 \end{align*}
We expand $2\pi i \, E^{-+}_{M,p} = K_1  +K_2  - K_3   -K_4$ where
\begin{align*}
K_1(x,y)   &= \left(   \frac{h(x)^{1/2}}{h(y)^{1/2}}[1+i\phi'(y)] -   (1+iM)\right)  \frac{1}{(y-x) + i (\phi(y)-\phi(x))} \\
K_2(x,y)   &=    (1+iM) \left( \frac{1}{(y-x) + i(\phi(y) - \phi(x))} - \frac{1}{(y-x) + i M(y+x)}\right)\\
K_3(x,y)   &= 
 \left(   \frac{h(y)^{1/2}}{h(x)^{1/2}}[1-i\phi'(x)] -   (1+iM)\right)  \frac{1}{(y-x) - i (\phi(y)-\phi(x))} \\
K_4(x,y)   &=  
   (1+iM) \left( \frac{1}{(y-x) - i(\phi(y) - \phi(x))} - \frac{1}{(y-x) - i M(y+x)}\right)
\end{align*}
Of course, these expressions hold only for $x\in(-a,0)$, $y\in(0,+a)$.  Outside this region, the kernels are zero.  We use $\opK_j$ to represent the integral operator associated to $K_j$, so $\opK_j f(x) = \int_0^a K_j(x,y)f(y)\, dy$ for $x\in (-a,0)$.

\medskip \noindent
{\underline{Estimate for $K_1$}:}   
First, an application of the Mean Value Theorem gives the existence of a value $c$ between $p'(x)$ and $-p'(y)$ such that
\begin{align*}
 h(x)^{1/2} - h(y)^{1/2}  &=   (1+ (M-p'(x))^2)^{1/4} - (1+ (M+p'(y))^2)^{1/4}  \\
 &=  \frac{-(M - c)}{2\,(1+ (M-c)^2)^{3/4}}  \cdot (p'(x)+p'(y)).
\end{align*}
From this it follows that $|h(x)^{1/2} - h(y)^{1/2}| \leq \| p' \|_\infty$.  By further using the expansion
\[
 \frac{h(x)^{1/2}}{h(y)^{1/2}}[1+i\phi'(y)] -   (1+iM) 
  =  \left( h(x)^{1/2} - h(y)^{1/2}\right) \frac{1+i\phi'(y) }{h(y)^{1/2}} + i \left(\phi'(y) -M \right)
\]
we may estimate
\begin{align*}
\left|  \frac{h(x)^{1/2}}{h(y)^{1/2}}[1+i\phi'(y)] -   (1+iM) \right| 
& \leq \| p'\|_\infty \cdot h(y)^{1/2} + |p'(y)| \\
& \leq  2\,(1+(|M|+\|p'\|_\infty)^2)^{1/4} \,  \| p' \|_\infty.
\end{align*}
Since $x < 0$, $y>0$, it then follows that 
\[
| K_1(x,y) |  \leq 2 \,(1+(|M|+\|p'\|_\infty)^2)^{1/4}  \,  \| p' \|_\infty \cdot \frac{1}{y-x},
\]
and so for $x\in(-a,0)$, 
\begin{align*}
|\opK_1 f(x)|  &\leq  \int_0^{+a} |K_1(x,y)| \, |f(y)| \, dy  \\
& \leq 2 \,(1+(|M|+\|p'\|_\infty)^2)^{1/4}  \,  \| p' \|_\infty \cdot \int_0^{+a} \frac{|f(y)| }{y-x} \, dt\\[.035in]
& =  2\, (1+(|M|+\|p'\|_\infty)^2)^{1/4}  \,  \| p' \|_\infty \cdot \pi \, \opH(\chi_{(0,+a)} \, |f|)(x)
\end{align*}
where $\opH$ is the Hilbert transform for the real line.  The Hilbert transform is bounded and has norm 1; the same holds for multiplication by $\chi_{(0,+a)}$.  
It follows that $\opK_1$ is bounded, and
\begin{equation}
\| \opK_1 \| \leq 2\, (1+(|M|+\|p'\|_\infty)^2)^{1/4}  \,  \| p' \|_\infty \cdot \pi.
\label{eqK1}
\end{equation}

\medskip\noindent{\underline{Estimate for $K_2$}:} 
Proceeding as before, we make the initial estimate
\begin{align*}
| K_2(x,y) | &= (1+M^2)^{1/2} \cdot 
 \frac{|p(y)- p(x)|}{    | (y-x) + i(\phi(y) - \phi(x))|} \cdot \frac{1}{ |(y-x) + i M(y+x)|} \\
&\leq (1+M^2)^{1/2} \cdot   \| p' \|_\infty \cdot   \frac{1}{y-x}.
\end{align*} 
Using the very same reasoning as before, we obtain
\begin{equation}
\| \opK_2 \| \leq(1+M^2)^{1/2} \cdot   \| p' \|_\infty \cdot \pi.
\label{eqK2}
\end{equation}

\medskip\noindent{\underline{Estimates for $K_3$, $K_4$}:}   The estimates for $K_3$ and $K_4$ follow the same reasoning as for $K_1$ and $K_2$ and lead to
\begin{align}
\| \opK_3 \| & \leq 2\, (1+(|M|+\|p'\|_\infty)^2)^{1/4}  \,  \| p' \|_\infty \cdot \pi
\label{eqK3}
\\
\| \opK_4 \| &\leq (1+M^2)^{1/2} \cdot   \| p' \|_\infty \cdot \pi.
\label{eqK4}
\end{align}

\medskip\noindent
Since $ \opE^{-+}_{M,p} = (\opK_1  +\opK_2  -\opK_3   -\opK_4)/(2\pi i)$, we then conclude from \eqref{eqK1}, \eqref{eqK2}, \eqref{eqK3}, \eqref{eqK4} that
\begin{align*}
\| \opE^{-+}_{M,p} \| &\leq \|p'\|_\infty \left(
2 (1+(|M|+\| p' \|_\infty)^2)^{1/4} + (1+M^2)^{1/2}
\right) \\
&\leq \|p'\|_\infty \, 3 \,(1+(|M|+\| p' \|_\infty)^2)^{1/2} 
\end{align*}
This establishes \eqref{eqA01}, and as mentioned, 
\eqref{eqA02} follows immediately.

\medskip
We have left to verify that $\opA_{M,p} - \opA_{M,0} = 
 \opE^{--}_{M,p} + \opE^{-+}_{M,p}+\opE^{+-}_{M,p}+ \opE^{++}_{M,p}$ is compact, and based on the preceding, it will be enough to demonstrate that $\opE_{M,p}^{-+}$ is compact.
We do this using Lemma~\ref{lemB2}.  

Take $\eps>0$.  Given that $M\neq 0$ and $p\in C^1([-a,+a])$ are fixed, and $p(0) = p'(0)=  0$, there exists $a'>0$ such that
\[
|p'(x)| \, 3 \,(1+(|M|+| p'(x) |)^2)^{1/2} < \eps
\] 
for $x\in [-a',+a']$.
From what already has been established, it follows that if $\opL_1$ is the integral operator on $[-a,+a]$ that has kernel
\[
L_1 (x,y) = E_{M,p}^{-+} (x,y) \chi_{(-a',+a')}(x) \chi_{(-a',+a')}(y),
\] 
then $\| \opL_1 \| < \eps$.  It remains to be seen that $\opL_2 \eqdef \opE_{M,p}^{-+} - \opL_1$ is compact. 

Recalling that already $E^{-+}_{M,p} (x,y)$ vanishes except when both $x<0$ and $y>0$, it follows that 
\[
L_2(x,y) = (A_{M,p}(x,y) - A_{M,0}(x,y))
[  \chi_{(-a,0)}(x)  \chi_{(0,+a)}(y)
-
  \chi_{(-a',0)}(x)  \chi_{(0,+a')}(y)
].
\]
In particular, $|x-y|>2a'$ when $L_2(x,y) \neq 0$, and a simple estimate shows 
\[
\int_{-a}^{+a}\int_{-a}^{+a}  |L_2(x,y)|^2 \, dx \, dy < \frac{a^2}{2\pi^2} 
 \left(\frac{  1 + (|M|+\|p'\|_\infty)^2)}{(2a')^2} +  \frac{1+M^2}{(2a')^2}\right)
\] 
which is finite.  So $\opL_2$ is Hilbert-Schmidt and therefore compact.
\end{proof}

\subsection{Proof that the remaining $\opA_{j,k}$ are compact.}  We have left to verify that the remaining pieces $\opA_{j,k}$ are compact.  For this we use the Lanzani result that the Kerzman-Stein operator for a continuously differentiable region is compact, and the fact that a Hilbert-Schmidt integral operator is compact.

\begin{proof}[Proof that $\opA_{0,0}$ is compact.]  To  see that $\opA_{0,0}$ is compact on $L^2(\bOm)$ we consider 
a smoothing of the corners of $\Om$ to get a continuously differentiable region $\widetilde\Om$ 
for which $\bOm \cap \supp \chi_0 = b\widetilde\Om \cap \supp \chi_0$.  Lanzani showed that the Kerzman-Stein operator $\widetilde\opA$ for $\widetilde\Om$ is compact on $L^2(b\widetilde\Om)$. 
Due to the  cutoff functions, we write without risk of ambiguity, 
$\opA_{0,0} = \opM_0 \circ \widetilde \opA \circ \opM_0$ 
where $\opM_0$ denotes multiplication by $\chi_0$ as earlier.  Since $\widetilde\opA$ is compact and $\chi_0$ is bounded, it follows that $\opA_{0,0}$ is compact.
\end{proof}

\begin{proof}[Proof that $\opA_{0,k}$ (for $k>0$) is compact.] 
Let $\bOm$ be expressed locally as a graph at $w_k$ as in the earlier paragraph.
 We write $\chi_k = \chi_k^- + \chi_k^+$ where $\chi_k^-$ takes value 1 on the part of $\bOm \cap \ov U_k $ for which $x\leq 0$ and where $\chi_k^+$ takes value 1 on the part for which $x> 0$. 
(Both $\chi_k^-$ and $\chi_k^+$ take value 0 outside $\ov U_k$.)
As in the previous paragraph, we smooth the corners $w_j$ of $\bOm$, $j\neq k$, and suitably extend $\bOm$ at $w_k$ so as to obtain continuously differentiable regions $\widetilde\Om^\pm$ for which 
$
\bOm \cap \supp (\chi_0 + \chi_k^\pm)  = b\widetilde\Om^\pm \cap \supp (\chi_0 + \chi_k^\pm).
$
Letting $\widetilde\opA^\pm$ be the Kerzman-Stein operator for $\widetilde\Om^\pm$ and $\opM^\pm_k$ be multiplication by $\chi_k^\pm$, we may write  without risk of ambiguity, $\opA_{0,k} =  (\opM_0 \circ \widetilde\opA^+ \circ \opM_k^+)  
  + (\opM_0 \circ \widetilde \opA^- \circ \opM_k^-)$.  So then $\opA_{0,k}$ is compact.
\end{proof}

\begin{proof}[Proof that $\opA_{j,0}$ (for $j>0$) is compact.]
The argument uses the same ideas as in the previous paragraph.  Following the notation described there, we write 
$\opA_{j,0} =  (\opM_j^+ \circ \widetilde \opA^+ \circ \opM_0)  
  + (\opM_j^- \circ \widetilde \opA^- \circ \opM_0)$.  So then $\opA_{j,0}$ is compact.
\end{proof}

\begin{proof}[Proof that $\opA_{j,k}$ (for $j, k>0, j\neq k$) is compact.] 
Since  $\ov U_j$ and $\ov U_k$ are nonintersecting and the supports of $\chi_j$ and $\chi_k$ are contained in $\bOm\cap \ov U_j$ and $\bOm \cap \ov U_k$, respectively, it follows that there is a positive constant $\delta_{j,k}$ for which $z\in \supp \chi_j$, $w\in \supp \chi_k$  implies $|z-w| > \delta_{j,k}$. A simple estimate shows
\[
\iint_{b\Om\times b\Om} |A_{j,k}(z,w) |^2 \, ds_z ds_w < \frac{1}{4\pi^2} \, \frac{4}{\delta_{j,k}^2} \, \mathrm{length}(b\Om)^2 < \infty
\]
so that $\opA_{j,k}$ is Hilbert-Schmidt and therefore compact.
\end{proof}

%
%
\section{Change in eigenvalues under continuously differentiable perturbation}

In this section we demonstrate how the spectrum of the Kerzman-Stein operator changes with respect to a local continuously differentiable perturbation.  We restrict our attention to curves that are continuously differentiable, so following Lanzani's result, the Kerzman-Stein operator is compact and has only point spectrum.  Naturally, the result can be extended to curves that are piecewise continuously differentiable provided the perturbations preserve the angles at the corners.

The result is modeled on the proof that Lanzani gave for compactness~\cite{lanzani}.

\begin{prop}
Let  $p_1, p_2$ be continuously differentiable functions on an interval $I$ for which $\| p'_j \|_\infty < 1$.  If  $\Gamma_j$ is the graph $z(x) = x + i p_j(x)$, $x\in I$, and if $\opA_j$ is the Kerzman-Stein operator for $\Gamma_j$, $j=1,2$, then when expressed as integral operators in terms of the graph parameter, $\| \opA_1 - \opA_2 \| \leq C \| (p_1 -p_2)' \|_\infty$.
\label{propB2}
\end{prop}
 
We note that it is not necessary for the interval to be bounded, and in fact, the statement holds when $I = (-\infty,+\infty)$.  The result still is a local result, however, because a general curve can be expressed as a graph only locally.

As in the proof of Lanzani's result, the proof of Proposition~\ref{propB2} uses ideas of Coifman, Macintosh and Meyer~\cite{CMM}.  For our application, we use the following simple generalization of their Lemma 4 (in \cite{CMM}).  The statement and proof also follow the presentation of Torchinsky~\cite[Theorem XVI.2.2]{torchinsky} and so we omit the proof.

\begin{lemma}
\label{thm:CMM extension} 
Let $\eta>0$ and $B, \vp_1,\vp_2$ be Lipschitz functions on $\mR$ so that $\|\vp_j'\|_\infty <\eta$, $j=1,2$. Let  $F_1$, $F_2$ be holomorphic functions on $B(0,\eta)\subset\mC$. If $T_B$ is defined on $L^2(\mR)$ by
\[
T_B f(x) = P.V. \int_\mR \frac{B(x)-B(y)}{(x-y)^2}
F_1 \Big(\frac{\vp_1(x)-\vp_1(y)}{x-y}\Big) 
F_2 \Big(\frac{\vp_2(x)-\vp_2(y)}{x-y}\Big) 
f(y)\, dy,
\]
then $T_B$ is well-defined and bounded. Moreover,
$ \|T_B f\| \leq c(\eta,F_1,F_2) \|B'\|_\infty \|f\| $. 
\end{lemma}

\begin{proof}[Proof of Proposition~\ref{propB2}]
Following Lanzani's notation, define $h_j(x) = (1+ p_j'(x)^2)^{1/2}$.  
The isometry $\Lambda_j: L^2(\Gamma_j) \rar L^2(I)$ given by
\[
(\Lambda_j f)(x) = f(x + ip_j(x)) \, h_j(x)^{1/2}
\]
enables us to express $\opA_j$ as an integral operator on $L^2(I)$ with kernel 
\begin{align*}
A_j&(x,y) \\  &=  \frac{1}{2\pi i} \frac{1}{h_j(x)^{1/2} h_j(y)^{1/2}} 
\left(
 \frac{h_j(x)[1+i p_j'(y)] }{(y-x) + i[p_j(y) - p_j(x)]} -
 \frac{h_j(y)[1-i p_j'(x)] }{(y-x) - i[p_j(y) - p_j(x)]} 
 \right) \! .
\end{align*}
We then expand 
\begin{align*}
2\pi i[A_1(x,y) - A_2(x,y)] = & \, E_1(x,y) + E_2(x,y) + E_3(x,y) \\
   & +E_4(x,y) +E_5(x,y) +E_6(x,y)
\end{align*}
where 
\begin{align*}
E_1(x,y) & = h_1(x)^{1/2}\,\frac{1 +ip_1'(y)}{h_1(y)^{1/2}} \,
\frac{1}{y-x} \cdot  \\ 
& \hspace{.8in}\left[
\left(1+ i\, \ds \frac{p_1(y) - p_1(x)}{y-x} \right)^{-1} 
- \left(1+ i\, \ds  \frac{p_2(y) - p_2(x)}{y-x}\right)^{-1}  \right] \\
E_2(x,y) & =  
 \, h_1(x)^{1/2} \left[ \frac{1+i p_1'(y)}{h_1(y)^{1/2}}  
-\frac{1+i p_2'(y)}{h_2(y)^{1/2}}  
 \right]
\, \frac{1}{y-x} \left( 1+ i\, \ds \frac{p_2(y)-p_2(x)}{y-x}\right)^{-1} \\
E_3(x,y) & = 
\left[
h_1(x)^{1/2} - h_2(x)^{1/2} \right] 
\frac{1+i p_2'(y)}{h_2(y)^{1/2}}  
\, \frac{1}{y-x}  \left( 1+ i\, \ds \frac{p_2(y)-p_2(x)}{y-x}\right)^{-1}
\end{align*}
and where $E_4(x,y) = \ov{E_1(y,x)}$, $E_5(x,y) = \ov{E_2(y,x)}$ and $E_6(x,y) = \ov{E_3(y,x)}$.
To prove Proposition~\ref{propB2} it therefore suffices to show 
\begin{align*}
\| \opE_1\| &\leq C_1 \| (p_1 - p_2)' \|_\infty\\[.02in]
\| \opE_2\| &\leq C_2 \| (p_1 - p_2)' \|_\infty \\[.02in]
\| \opE_3\| &\leq C_3 \| (p_1 - p_2)' \|_\infty 
\end{align*}
where $\opE_1, \opE_2, \opE_3$ are the integral operators associated to $E_1, E_2, E_3$, respectively.

After writing 
\begin{align*}
E_1(x,y) & = 
h_1(x)^{1/2}\,\frac{1 +ip_1'(y)}{h_1(y)^{1/2}} 
 \frac{ (p_1 - p_2)(y) - (p_1-p_2)(x)}{ i (y-x)^2} 
\left(1+ i \ds \frac{p_1(y) - p_1(x)}{y-x} \right)^{-1}  \\
& \hspace{2.8in} \cdot  \left(1+ i \ds  \frac{p_2(y) - p_2(x)}{y-x}\right)^{-1}  
\end{align*}
we obtain $ \| \opE_1\| < C_1 \| (p_1 - p_2)' \|_\infty$ from an application of Lemma~\ref{thm:CMM extension} that uses 
\begin{align*}
B(x) \, &= (p_1-p_2)(x)\\
 \vp_1(x) &= p_1(x)\\
 \vp_2(x) &= p_2(x)\\
F_1(z) = F_2(z) &= (1+iz)^{-1}.
\end{align*}
We also use the fact that multiplication by $h_1(x)^{1/2}$ and by $(1+ip_1'(y))/h_1(y)^{1/2}$ are both operators with norm less than $\sqrt 2$ (since $\| p_1'\|_\infty<1$) and that multiplication by a characteristic function (in this case for interval $I$) is an operator with norm 1.

For $E_2$, we first make a preliminary estimate
\[
\left| 
\frac{1+i\alpha}{\sqrt[4]{1+\alpha^2}}
-
\frac{1+i\beta}{\sqrt[4]{1+\beta^2}} 
\right|
\leq \sqrt 2 \, |\alpha - \beta| 
\]
which can be justified using calculus applied separately to real and imaginary parts.
If we define
\[
g(y)  = 
\frac{1+ip_1'(y)}{h_1(y)^{1/2}}  
-
\frac{1+ip_2'(y)}{h_2(y)^{1/2}}, 
\]
then $|g(y) | \leq \sqrt 2 \, |(p'_1-p'_2)(y)| \leq \| (p_1 - p_2)'\|_\infty$.
With this observation, we have $\|\opE_2\| < c \|(p_1 - p_2)' \|_\infty$ after an application of 
Lemma~\ref{thm:CMM extension}
that uses 
\begin{align*}
B(x) \, &= x\\
 \vp_2(x) &= p_2(x)\\
F_1(z) &= 1 \\ 
F_2(z) &= (1+iz)^{-1}.
\end{align*}
The factor $\|(p_1 - p_2)' \|_\infty$ arises from the initial multiplication by $g(y)$.  (The argument again also uses the fact that multiplication by $h_1(x)^{1/2}$ is an operator with bounded norm.)

The reasoning for $\opE_3$ is similar to the reasoning for $\opE_2$ except we use  
$|\sqrt[4]{1+\alpha^2}  - \sqrt[4]{1+\beta^2}| \leq |\alpha-\beta|$
in order to show $|h_1(x)^{1/2} - h_2(x)^{1/2}| \leq \| (p_1-p_2)'\|_\infty$.
\end{proof}

%
%

\section{Final Remarks}

In closing, we reiterate the observation that our analysis of the spectrum of the Kerzman-Stein operator suggests that the global geometry of a curve is reflected more in the point spectrum, and the local geometry (corners, etc.) is reflected more in the continuous spectrum.  It would be interesting to pursue this question further.  In particular, it would be interesting to have a complete description of the spectrum for a triangle; i.e., does the Kerzman-Stein operator then have any eigenvalues?

As well, we wonder if Theorem 2 can be shown to be true for regions with boundaries that are Lipschitz but not piecewise continuously differentiable?  Also, is it possible to extend Theorems 1 and 2 to allow for the case of a cusp?  The authors intend to return to these problems in subsequent work. 

\bibliographystyle{plain}
\bibliography{refs.bib}

\end{document}